\documentclass{article}
\usepackage{geometry}                
\usepackage{mathtools}
\usepackage{amssymb}
\usepackage{amsmath}
\usepackage{amsthm,dsfont,yfonts}
\usepackage{rotating}
\usepackage{graphicx,pspicture}
\usepackage{array}
\usepackage{calc}
\usepackage{soul}
\usepackage[all]{xy}
\usepackage{amsthm}
\usepackage{lmodern}
\usepackage{color}
\usepackage{pstricks,pst-node,pst-tree}
\usepackage{graphics,graphicx}
\usepackage[british]{babel}
\usepackage{fancyhdr}
\usepackage{color}
\usepackage{pinlabel}
\usepackage{enumerate}
\usepackage{footmisc}
\usepackage{float}

\newtheorem{theorem}{Theorem}[section]
\newtheorem{theorem_nn}{Theorem}
\newtheorem{cor}[theorem]{Corollary}
\newtheorem{lemma}[theorem]{Lemma}
\newtheorem{prop}[theorem]{Proposition}

\theoremstyle{definition}
\newtheorem{definition}[theorem]{Definition}
\newtheorem{example}[theorem]{Example}

\newtheorem{question}{Question}

\newcommand{\Aut}{\operatorname{Aut}}

\newcommand{\Sym}{\operatorname{Sym}}

\newcommand{\St}{\operatorname{St}}
\newcommand{\rst}{\operatorname{rst}}

\definecolor{magenta}{RGB}{203,0,150}
\definecolor{blueish}{RGB}{0,35,211}
\definecolor{greenish}{RGB}{0,100,20}

\newcommand{\h}{\hspace{2mm}}  

\usepackage{makeidx}
\makeindex

\usepackage[noindentafter]{titlesec}
\title{Conjugacy Growth and Conjugacy Width\\ of Certain Branch Groups}
\date{\today}
\author{Elisabeth Fink\footnote{This work is supported by the ERC starting grant 257110 ``RaWG''}}
  
\begin{document}

\selectlanguage{british}

\maketitle

\begin{abstract}
The conjugacy growth function counts the number of distinct conjugacy classes
in a ball of radius $n$. We give a lower bound for the conjugacy growth of certain branch groups, among them the Grigorchuk
group. This bound is a function of intermediate growth.
We further proof that certain branch groups have the property that
every element can be expressed as a product of uniformly boundedly many conjugates of the generators. We call this property
bounded conjugacy width. We also show how bounded conjugacy width relates to other algebraic properties of groups and apply these
results to study the palindromic width of some branch groups.
\end{abstract}

\begin{center}
\textbf{\small{Keywords: branch groups, conjugacy growth.\\MSC classification: 20F65, 20F69.}}
\end{center}

\section{Introduction}

The conjugacy growth function of a
group was first introduced by I. K. Babenko in \cite{babenko} to
study geodesic growth of Riemannian manifolds. It counts the number of distinct conjugacy classes in a ball of radius $n$. This
function had already intensively been studied for manifolds, among others by G. Margulis (\cite{margulis}) who obtained results in
the
case of negatively curved manifolds. These results have been generalized by T. Roblin (\cite{roblin}) to any quotient of a CAT(-1)
metric space and further by I. Gekhtman (\cite{gekhtman}) to elements of mapping class groups. It was shown by E. Breuillard and
Y. Cornulier in \cite{breuillard_corn} that the conjugacy growth
function
of a solvable group is either polynomially bounded or exponential. Recently, M. Hull and D. Osin proved in
\cite{hull_osin}
that for any 'sensible' function $f(n)$, there exists a finitely generated group such that it has
conjugacy growth exactly $f(n)$. The paper \cite{sapir_guba} gives a summary of examples and conjectures
concerning conjugacy growth.  

\medskip

In this paper we study the conjugacy growth of a wide class of branch groups, among them the Grigorchuk group. The following
theorem states that for many classes of branch groups this conjugacy growth is bounded from below by an intermediate function.

\begin{theorem_nn}[=Theorem \ref{thm_conj_branch}]
Let $G$ be a finitely generated regular branch group acting on a $d$-regular rooted tree. Then the conjugacy
growth
function $f(n)$ of $G$ satisfies
\[e^{n^\sigma} \precsim f(n),\] where $0<\sigma<1$,
which can be made explicit depending on the group.
\end{theorem_nn}

It would be desirable to obtain upper bounds for the conjugacy growth as well. In particular for the Grigorchuk group, whose
word growth is still only known to be bounded from below by $e^{n^{0.521}}$ as shown by J. Brieussel in his thesis
\cite{brieussel_thesis} and from above by $e^{n^{0.767}}$ as shown by L. Bartholdi in \cite{bartholdi_upper}. However, since there
exist branch groups of exponential word growth, one would need to restrict the groups under investigation to obtain interesting
results.

\medskip

In the second part of this paper we show that certain branch groups have the property that every element can be written as a
product of uniformly boundedly many generators. Properties like these have
been studied under the name bi-invariant metrics for various groups (see for example \cite{brandenbursky_gal_kedra_marcinkowski}
and \cite{burago_ivanov_polterovich}). The same property has been studied
under the name \emph{reflection length} in Coxeter groups in \cite{duszenko} and \cite{mccammond_peterson}. It will be shown
that bounded conjugacy width implies a number of other algebraic properties. We obtain the following result about branch groups:

\begin{theorem_nn}[=Theorem \ref{thm_conj_width}]
Let $G$ be a just infinite branch group that contains a rooted element and that has finite commutator width. Then $G$ has
bounded conjugacy width.
\end{theorem_nn}

We will show in Subsection \ref{subsec_bcw} that the conditions to be just infinite and to have finite commutator width are
necessary. In particular, we will give examples of groups which have finite commutator width but which have unbounded conjugacy
width.

\medskip

We can apply our results to the study of the palindromic
width of the Grigorchuk group. A palindrome is a word that reads the same right-to-left as left-to-right. It has been studied by
various authors over the last decade, whether for a group $G$ there exists a uniform bound $N$, such that every element of $G$
can be expressed as a product of at most $N$ palindromes (see for example \cite{bardakov_nilpotent}, \cite{bardakov},
\cite{me_palindromic}, \cite{me_andreas} or \cite{rileySale}). It has been shown in \cite{me_andreas} that if a group is just
infinite, then it has finite palindromic width with respect to some finite generating sets. Here we complete this picture for the
Grigorchuk group and prove
that it has finite palindromic width with respect to all generating sets.

\medskip

\textbf{Acknowledgements.} I would like to thank Laurent Bartholdi for fruitful discussions on both topics of this paper and Anna
Erschler for some helpful suggestions on the conjugacy width problem. In particular I am very grateful to the referee for a
careful reading of this text and for pointing out some mistakes in a preliminary version of this paper.
\section{Branch Groups}

In this section we will recall some of the notation and definitions for branch groups from \cite{Bartholdi} and
\cite{Segal_finiteImages}.

\subsection{Trees}\label{subsec_trees}

A \emph{tree} is a connected graph which has no non-trivial cycles. If $T$ has a distinguished \emph{root} vertex $r$
it is called a \emph{rooted tree}. The distance of a vertex $v$ from the root is given by the length of the path from
$r$ to $v$ and called the \emph{norm} of $v$. The number \[d_v = | \{e \in E(T): e=\left(v_1, v_2\right), v = v_1
\textnormal{ or } v=v_2\}|\] is called the \emph{degree} of $v \in V(T)$. The tree is called \emph{spherically
homogeneous} if vertices of the same norm have the same degree. Let $\Omega(n)$ denote the set of vertices of distance
$n$ from the root. This set is called the \emph{$n$-th level} of $T$. A spherically homogeneous tree $T$ is determined
by a
finite or infinite sequence $\bar{l}=\left\{l_n\right\}_{n=0}$ where $l_n+1$ is the degree of
the vertices on level $n$ for $n \geq 1$. The root has degree $l_0$. Hence each level $\Omega(n)$ has $\prod_{i=0}^{n-1}
l_i$ vertices. Let us denote this number by $m_n = |\Omega(n)|$. We denote such a tree by $T_{\bar{l}}$. A tree is
called \emph{regular} if $l_i = l_{i+1}$ for all $i \in \mathbb{N}_0$. Given a spherically homogeneous tree $T$ we
denote by $T[n]$ the finite tree where all vertices have norm less or equal to $n$ and write $T_v$ for the subtree of
$T$ with root $v$.
For all vertices $v,u \in \Omega(n)$ we have that $T_u \simeq T_v$. Denote a tree isomorphic to $T_v$ for $v \in
\Omega(n)$ by $T_n$. This will be the tree with defining sequence $\left(l_n, l_{n+1}, \dots \right)$. To each sequence
$\bar{l}$ we associate a sequence $\left\{X_n\right\}_{n \in \mathbb{N}_0}$ of alphabets where $X_n=\left\{v_1^{(n)},
\dots, v_{l_n}^{(n)}\right\}$ is an $l_n$-tuple so that $|X_n|=l_n$.  A path beginning at the root of length $n$ in
$T_{\bar{l}}$ is identified with the sequence ${x_1,\dots, x_i, \dots, x_n}$ where $x_i \in X_i$ and infinite paths are
identified in a natural way with infinite sequences. Vertices will be identified with finite strings in the alphabets
$X_i$. Vertices on level $n$ can be written as elements of $Y_n  = X_0 \times \dots \times X_{n-1}$. Alphabets induce
the lexicographic order on the paths of a tree and therefore on the vertices.

\subsection{Automorphisms}

An \emph{automorphism} of a rooted tree $T$ is a bijection from $V(T)$ to $V(T)$ that preserves edge incidence and the
distinguished root vertex $r$. The set of all such bijections is denoted by $\Aut(T)$. This group acts as an
imprimitive permutation group on the set $\Omega(n)$ of vertices on level $n$ for each $n \geq 2$. Consider an element
$g
\in \Aut(T)$. Let $y$ be a letter from $Y_n$,
hence a vertex of $T[n]$ and $z$ a vertex of $T_n$. Then $g(y)$ induces a vertex permutation of $Y_n$. If we
denote
the image of $z$ under $g$ by $g_y(z)$ then \[g(yz)= g(y)
g_y(z).\]

\medskip

With any group $G \leq \Aut(T)$ we associate the subgroups \[\St_G(u)=\left\{g \in G: g(u)=u\right\},\] the
\emph{stabilizer} of a vertex $u$. Then the subgroup \[\St_G(n)=\bigcap_{u \in \Omega(n)} \St_G(u)\] is called the
\emph{$n$-th level stabilizer} and it fixes all vertices on the $n$-th level. Another important class of subgroups
associated with $G \leq \Aut(T)$ consists of the \emph{rigid vertex stabilizers} \[\rst_G(u)=\left\{g \in G: \forall v
\in V(T) \setminus V(T_u): g(v)=v\right\}.\] Informally speaking, $\rst_G(u)$ fixes everything outside the subtree $T_u$ with
root $u$. The subgroup \[\rst_G(n)= \prod_{u \in \Omega(n)} \rst_G(u)\] is called the \emph{$n$-th level rigid stabilizer}.
Obviously $\rst_G(n) \leq \St_G(n)$.

\begin{definition}
Let $G$ be a subgroup of $\Aut(T)$ where $T$ is a spherically homogeneous rooted tree. We say
that $G$ acts on $T$ as a \emph{branch
group} if it acts transitively on the vertices of each level of $T$ and $\rst_G(n)$ has finite index for all $n \in
\mathbb{N}$.
\end{definition}
The definition implies that branch groups are infinite and residually finite groups. We can specify an automorphism $g$
of $T$ that fixes all vertices of level $n$ by writing $g = \left(g_1, g_2, \dots, g_{m_n}\right)_n$ with $g_i \in
\Aut\left(T_n\right)$ where the subscript $n$ of the brackets indicates that we are on level $n$. Each automorphism can
be written as $g = \left(g_1, g_2, \dots, g_{m_n}\right)_n \cdot \alpha $
with $g_i \in \Aut\left(T_n\right)$ and $\alpha$ an element of $\Sym\left(l_{n-1}\right) \wr \dots \wr
\Sym\left(l_0\right)$. Automorphisms acting only on level $1$ by permutation are called \emph{rooted automorphisms}. We
can identify those with elements of $\Sym\left(l_0\right)$.

\medskip

\begin{definition}\label{def_branch}An algebraic
description of a branch group is given by the existence of a sequence of \emph{branching subgroups}. In particular,
we say $G$ is a \emph{branch group} if there exist two decreasing sequences of subgroups $\left(L_i\right)_{i \in \mathbb{N}_0}$
and
$\left(H_i\right)_{i \in \mathbb{N}_0}$ and a sequence of integers $\left(k_i\right)_{i \in \mathbb{N}_0}$ such that $L_0 = H_0 =
G,
k_0=1$, \[\bigcap_{i \in \mathbb{N}_0} H_i=1\] and for each $i$
\begin{enumerate}
 \item $H_i$ is a normal subgroup of $G$ of finite index,
 \item $H_i$ is a direct product of $k_i$ copies of the subgroup $L_i$, in other words there are subgroups $L_i^{(1)}, \dots,
L_i^{(k_i)}$ of $G$ such that \[H_i = L_i^{(1)} \times \cdots \times L_i^{(k_i)}\] and each of the factors is isomorphic to $L_i$,
 \item $k_i$ properly divides $k_{i+1}$, i.e. $m_{i+1} = k_{i+1}/k_i \geq 2$, and the product decomposition of $H_{i+1}$ refines
the product decomposition of $H_i$ in the sense that each factor $L_i^{(j)}$ of $H_i$ contains $m_{i+1}$ of the factors of
$H_{i+1}$, namely the factors $L_{i+1}^{(l)}$ for $l=(j-1)m_{i+1}+1,\dots ,jm_{i+1}$,
 \item conjugations by the elements in $G$ transitively permute the factors in the product decomposition.
\end{enumerate}
\end{definition}

These two definitions are in general not equivalent as stated in \cite{Bartholdi_bg}. However, if a group is a branch group by
the geometric definition, then one can easily recover the branching structure that is required in the algebraic definition. For
more details, please see \cite{Bartholdi_bg}.
\medskip

In many cases the structure of branch groups becomes more accessible if we impose another condition. 

\begin{definition}\label{def_branch_reg}
A branch group $G$ acting on a rooted regular tree is called \emph{regular branch} over its normal subgroup $H$ if $H$ has finite
index
in $G$, $(H,\dots, H)_1 \leq H$ and if moreover the last inclusion is of finite index.
\end{definition}

In particular, this last definition allows us to study branch groups via so-called self-similarity arguments.

\section{The Grigorchuk Group}

The Grigorchuk group, which was introduced in \cite{grigor_1} by R. Grigorchuk, is defined via its action on a rooted binary tree. It is
generated by four automorphisms. The first one, $a$,
swaps the two top subtrees. The other three are defined recursively as
\begin{equation}\label{eq_sub}b=(a,c)_1, \quad c=(a,d)_1, \quad d=(1,b)_1.\end{equation} It is helpful to picture them via
their actions on the binary tree (see Figure \ref{fig_grig_gen}). We define the Grigorchuk group $\Gamma$ as $\Gamma =
\left<a,b,c,d\right>$.

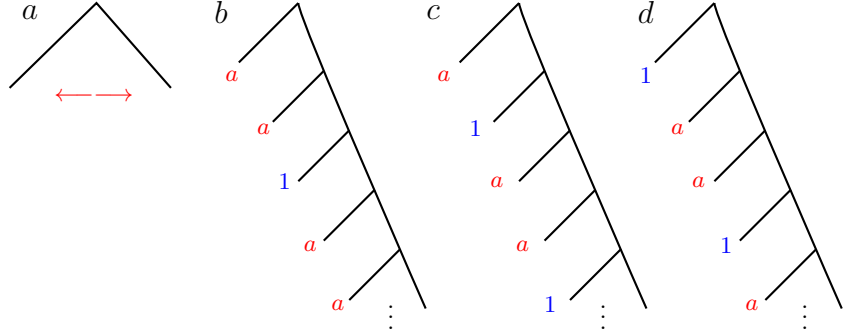
\begin{figure}[H]
\begin{center}
\labellist
\small
\pinlabel \textcolor{black}{\large{$a$}} at 10 140
\pinlabel \textcolor{black}{\large{$b$}} at 100 140
\pinlabel \textcolor{black}{\large{$c$}} at 200 140
\pinlabel \textcolor{black}{\large{$d$}} at 300 140
\pinlabel \textcolor{red}{{$a$}} at 105 110
\pinlabel \textcolor{red}{{$a$}} at 120 85
\pinlabel \textcolor{blue}{{$1$}} at 130 60
\pinlabel \textcolor{red}{{$a$}} at 142 30
\pinlabel \textcolor{red}{{$a$}} at 155 2

\pinlabel  \textcolor{red}{$\longleftarrow \longrightarrow$} at 40 100

\pinlabel \textcolor{red}{{$a$}} at 205 110
\pinlabel \textcolor{blue}{{$1$}} at 220 85
\pinlabel \textcolor{red}{{$a$}} at 230 60
\pinlabel \textcolor{red}{{$a$}} at 242 30
\pinlabel \textcolor{blue}{{$1$}} at 255 2

\pinlabel \textcolor{blue}{{$1$}} at 300 110
\pinlabel \textcolor{red}{{$a$}} at 315 85
\pinlabel \textcolor{red}{{$a$}} at 325 60
\pinlabel \textcolor{blue}{{$1$}} at 338 30
\pinlabel \textcolor{red}{{$a$}} at 350 2
\pinlabel {$\vdots$} at 375 0
\pinlabel {$\vdots$} at 280 0
\pinlabel {$\vdots$} at 180 0
\endlabellist
\includegraphics[scale=0.8]{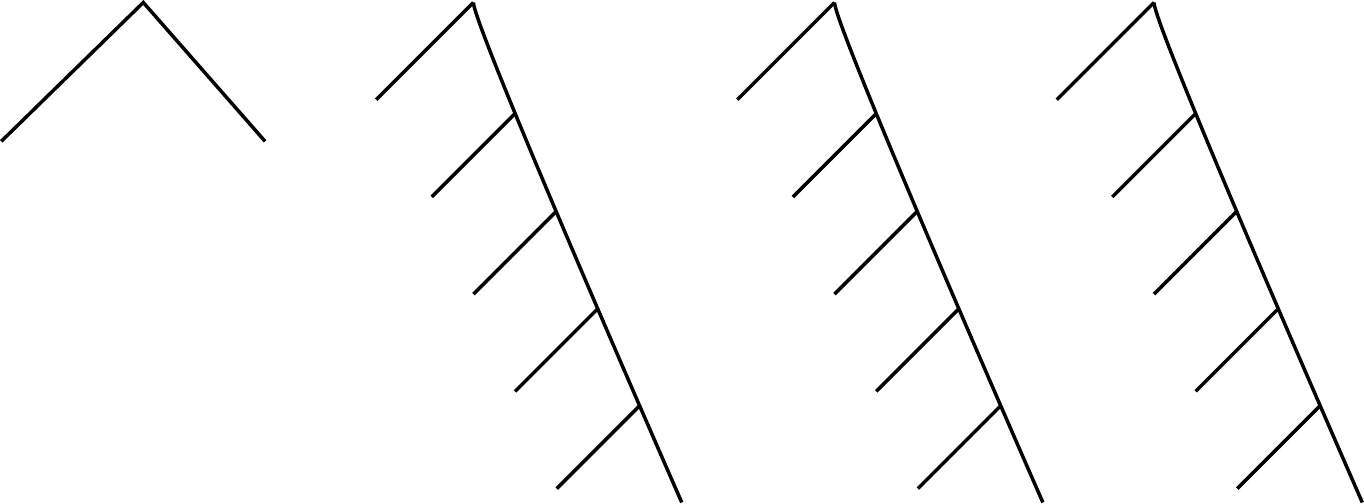}
\end{center}
\caption{Actions of the generators of the Grigorchuk group depicted on the binary tree.}
\label{fig_grig_gen}
\end{figure}

The Grigorchuk group has become well known due to its property of having intermediate word growth. It was further also the first
example of an amenable group that is not elementary amenable. The following proposition lists some of the many interesting properties of
this group.

\begin{prop}\label{prop_grig_prop}
The Grigorchuk group has the following properties:
 \begin{enumerate}
  \item It is a just infinite branch group.
  \item It is generated by three elements.
  \item \label{part_rel}It is infinitely presented, but in particular we have the following relations:
  \[a^2=b^2=c^2=d^2=1, \quad [b,c]=[b,d]=[c,d]=1.\]
  \item It has intermediate word growth.
 \end{enumerate}
\end{prop}

The relations in \ref{prop_grig_prop}(\ref{part_rel}) help to understand the structure of words in the Grigorchuk group. In
particular, we can deduce some normal
form:

\begin{lemma}
Any element $g \in \Gamma$ can be written as \begin{equation}\label{eq_grig_el}a^{\epsilon}*a*a*...*a*^{\delta},\end{equation}
where
$*$ stands for either $b,c$ or $d$ and $\epsilon, \delta \in \left\{0,1\right\}$.
\end{lemma}

\begin{proof}
This follows immediately from Proposition \ref{prop_grig_prop}(\ref{part_rel}).
\end{proof}

\section{Conjugacy Growth}\label{sec_conj_growth}

In this section we study the conjugacy growth function of branch groups acting on a regular rooted tree and then treat the special
case of the Grigorchuk group. We will consider this group separately in Subsection \ref{subsec_bcw_grig} to get a better
estimate than the one which would follow from
the general result about branch groups. Our approach for branch groups in general uses its finite index branching subgroup. It
emerges from
the work of M. Hull and D. Osin in \cite{hull_osin} that there exists a group with exponential conjugacy growth, but it has an
index
$2$ subgroup which has only $2$ conjugacy classes. We emphasize that in our approach we consider the lengths of the words in the
branching
subgroup in the word metric coming from the group itself and conjugation will also be considered in the whole group.

\medskip

Let $G$ be any finitely generated group. We will for the rest of this paper denote the conjugate of an element $g \in G$ by
another element $x \in G$ as
$g^x=x^{-1}gx$ and commutators by $[x,y]=x^{-1}y^{-1}xy, x,y \in G$. Further, we denote by $l_X(g)$ the word length of an element
$g$ in the
generators of the group. This word length
depends on the chosen generating set $X$. If it is clear which generating set we will refer to, we will omit the subscript $X$
and simply denote the length of $g$ by $l(g)$.

\medskip

A function $f: \mathbb{R}_+ \rightarrow \mathbb{R}_+$ is \emph{dominated} by $g: \mathbb{R}_+ \rightarrow \mathbb{R}_+$, written
$f \precsim g$, if there is a constant $C \in \mathbb{R}_+$ such that $f(n) \leq g(Cn)$ for all $n \in \mathbb{R}_+$. Two
functions are \emph{equivalent}, denoted by $f \sim g$, if $f \precsim g$ and $g \precsim f$.

\begin{definition}
Let $X$ be a generating set for a group $G$ and denote for $c \in G$ by $[c]$ the conjugacy class of $c$. The \emph{conjugacy
growth
function} is defined as \[f_X(n)=|\left\{[c]\h
| \h l_X(c)\leq n\right\}|.\] In words, the conjugacy growth function counts the number of distinct conjugacy classes within a
ball
of radius $n$. 
\end{definition}

This definition also depends on the chosen generating set. However, one can easily see that a change of the generating set does
not change the equivalence class of the conjugacy growth function. It is clear that this function is bounded from above by the
word growth $\gamma_G(n)$ of a group $G$. 

\subsection{Regular branch groups}

We show that if a branch group $G$ acts on a regular rooted tree, then its conjugacy growth function is bounded from below by a
function equivalent to $e^{n^\sigma}$ for some $0<\sigma<1$, which depends on the group.

\medskip

We first establish that the number of conjugacy classes in a branch group is in fact unbounded.

\begin{lemma}\label{lemma_conj_unbounded}
Let $G$ be a branch group, then $G$ has infinitely many distinct conjugacy classes.
\end{lemma}

\begin{proof}
Branch groups are residually finite groups, with a sequence of filtration subgroups given by $\left\{st_G(n)\right\}$. Hence there
exists a
sequence $\left\{n_i\right\}_{i \in \mathbb{N}}, n_i \in \mathbb{N}$, such that we can find elements $g_{i} \in G$ such that
$g_{i} \in st_G(n_i) \setminus st_G(n_{i+1})$. Any two
elements of the sequence $\left\{g_i\right\}_{i \in \mathbb{N}}$ are not conjugate.
\end{proof}

\begin{theorem}\label{thm_conj_branch}
Let $G$ be a finitely generated regular branch group acting on a $d$-regular rooted tree. Then the conjugacy
growth
function $f(n)$ of $G$ satisfies
\[e^{n^\sigma} \precsim f(n),\] where $0<\sigma<1$,
which can be made explicit depending on the group.
\end{theorem}

\begin{proof}
The group $G$ is by hypothesis branch, so we have the regular branching structure from Definition \ref{def_branch_reg}. The
branching subgroup $H$ has finite index in $G$ hence $H$ and $(H, \dots, H)_1$ are finitely generated as well. Say $(H, 1, \dots,
1)_1$ is
finitely generated by elements $\left\{y_1,
\dots, y_s\right\}$ and so $\left(H, \dots, H\right)_1$ is generated by $\left\{y_1, \dots, y_{l}\right\}$ with $l=sd$. If $X$ is
the finite generating set of $G$, then we denote by $M$ the maximum over all word lengths of the
$y_i, $\[M=\max_{j=1, \dots,
l}\left\{l_X(y_j)\right\}.\] In words, it takes at most a word of length $M$ in the letters $X$ to write each generator of each
copy of $H$ on each of
the $d$ subtrees of
level $1$. Figure \ref{fig_recursion} depicts this idea of finding multiple copies of $H$ on the first level which are contained
in $H$.
\begin{figure}[H]
\begin{center}

\labellist
\pinlabel \textcolor{blue}{\large{$H$}} at 30 20
\pinlabel \textcolor{blue}{\large{$H$}} at 110 20
\pinlabel \textcolor{blue}{\large{$H$}} at 280 20
\endlabellist

\includegraphics[scale=0.8]{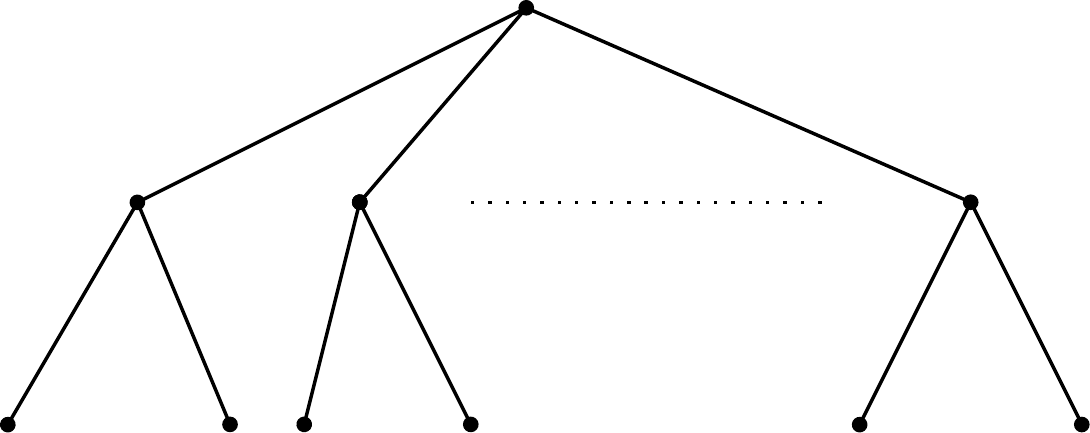}
\caption{Self-similarity of $H$.}
\label{fig_recursion}
\end{center}
\end{figure}
Now we count the number of distinct conjugacy classes in a ball of radius $d\cdot M \cdot n$. If an element $g=\left(g_1,\dots,
g_d\right)_1$ is
in
the first level stabilizer, it means that it does not permute any of the first level vertices. Then a conjugate of $g$ by
$h=\left(h_1,
\dots, h_d\right)_1 \cdot \tau$ has the following form:
\[g^h = \left(g_{\tau(1)}^{h_{\tau(1)}}, \dots, g_{\tau(d)}^{h_{\tau(d)}}\right)_1,\] where $\tau \in \Sym(d)$ is the
permutation of the first level vertices, coming from the element $h$. This allows us to apply a self-similarity argument. We undercount the number of conjugacy
classes
by assuming that we
write a word of length $n$ on each of the subtrees. Further, we need to consider that all permutations of the subtrees of the
first level are possible by conjugation. However,
there are at most as many different permutations of the first level vertices in $G$ as there are elements in $G/st_G(1)$. Denote
the index of $st_G(1)$ in $G$ by $K$.
So we get a recursive formula
\[f\left(d\cdot M \cdot n\right) \geq \frac1{T\cdot K}f(n)^d.\]
Here the product $d\cdot M \cdot n$ comes from writing $d$ copies of length $n$, but for each word of length one in each of these
copies
we need at most $M$ letters. Hence $d \cdot M \cdot n$ is an upper bound for writing a word of length $n$ on each of the $d$
subtrees. The
factor $1/(TK)$ is explained as follows: We only look at the words that stabilize level $1$. The
subgroup $st_G(1)$ has index $K$ which implies that there exists a constant $T$ such that $|B_G(n)\cap st_G(1)| \geq
\frac{B_G(n)}{T}$ for $n \geq K$. The factor $1/K$ comes from the fact that the $d$ subtrees can be permuted
by conjugation. However, as $st_G(1)$ has index $K$, there are at most $K$ different such permutations. The power $d$ of $f(n)$
indicates that each different combination of conjugacy classes on the subtrees gives a different element. We repeat this with
$Q=TK$ and get
\[f\left((dM)^in\right) \geq \frac{1}{Q^{1+d+d^2+\dots+d^{i-1}}} \cdot f(n)^{d^i}\] which is further
\[f\left((dM)^in\right) \geq \frac{1}{Q^{d^i}} \cdot f(n)^{d^i}=
\left(\frac{f(n)}{Q}\right)^{d^i}.\] We now start with $n=N$ such that $f(N)>Q$, as shown to be possible in Lemma
\ref{lemma_conj_unbounded}, to obtain
\[f\left((dM)^i N\right) \geq C^{d^i}, \] where $C>1$ is a constant. Rewritten with $k=(dM)^i$ we get
\[f(kN) \geq C^{k^\sigma},\] with $\sigma = \frac{\log (d)}{\log(dM)}$ and so \[f(k) \succsim e^{k^\sigma}.\]
\end{proof}

As we will see in the next section, this bound is rather general. By knowing more about the structure of a specific branch
group, this lower bound can in some cases be improved significantly. A similar approach can also be applied to branch groups which
are not regular. However, in such a case, the constants $M, K$ and $d$ in the proof above would be different on every level of
the tree. In
fact, they might not even follow the same recursion. Hence it appears rather difficult to find bounds for the conjugacy growth of
such a branch group with the approach that we have taken above. Examples of branch groups which are not regular  first emerged
in \cite{Segal_subgroupGrowth} and were later studied in \cite{me_branchgroups}.

\medskip

Of course an upper bound for the conjugacy growth would be very interesting. This depends apriori heavily on the group, as
some branch groups have intermediate word growth, others have exponential word growth. We however suspect, that for most branch
groups an upper bound for the conjugacy growth will not be very different than the one for the word growth.

\subsection{Grigorchuk group}\label{subsec_bcw_grig}

We could now apply Theorem \ref{thm_conj_branch} to the Grigorchuk group. This group is acting on a binary tree, so we
immediately get $K=d=2$. For reasons that we do not want to elaborate here, we have $M=24$ with the standard generating set, hence
yielding a lower
bound equivalent to $e^{n^{0.179}}$ for the conjugacy growth function. However, by studying the structure of the Grigorchuk group
more carefully, we obtain a better estimate. First we cite an auxiliary lemma that we
will need in the proof. We fix for the rest of this section the group $\Gamma$ as the Grigorchuk group.

\begin{lemma}\label{lemma_subwords}
We have a recursive formula:
\[\left(B_{\Gamma}\left(n\right), B_{\Gamma}\left(m\right)\right)_1 \subseteq B_{\Gamma}(2(n+m)).\] In words, with a word of
length $2(n+m)$ we can
write any combination of words of lengths $n$ respectively $m$ on the two top subtrees. 
\end{lemma}

\begin{proof}
To see this, we use the standard form of a word in \eqref{eq_grig_el}. We will show that it is possible to write any word
of
length $m$ on the right subtree without influencing the left subtree too much. 
Denote the two subwords by $w_0$ and $w_1$.
Obviously $w_0$ and $w_1$ again have the form in \eqref{eq_grig_el}. To write an $a$ as the first letter of $w_1$, we choose $w$
to begin with $aca$. To write any of the other generators, we use the substitution rules from \eqref{eq_sub}. So to get $ab$ in
$w_1$ we write $acad$:
\[w=acad \longrightarrow \left(1, ab\right)_1\]
\[w=acab \longrightarrow \left(1, ac\right)_1\]
\[w=acac \longrightarrow \left(1, ad\right)_1.\]
It is important to choose $aca$ in $w$ if
we want to write $a$ in $w_1$. This will then leave us with a word $a^{\epsilon} d\dots ad^{\delta}$ on the left subtree, with
$\epsilon, \delta \in \left\{0,1\right\}$.
Now $ad$ has order $4$. So
in the worst case we are left with $adad$ on the left subtree: \[w=\dots \longrightarrow \left(adad,
w_1\right)_1.\] However, we
want to choose
the
word $w_0$ freely. So if $w_0$ does not start with the word coming from writing $w_1$, we simply slightly modify our choice of
letters for
$w$ to write $w_1$. This can be done such that it does not affect $w_1$, but leaves us with $adac$, $acac$ or $acad$:
\[w=\dots *aba*aca \longrightarrow \left(acad, w_1\right)_1\]
\[w=\dots *ba*aba \longrightarrow  \left(acac, w_1\right)_1\]
\[w=\dots *aca*aba \longrightarrow  \left(adac, w_1\right)_1\]
If however one of the first occurrences in $w_1$ is the letter $b$, then we put $ad$ into $w$ at the beginning:
\[w=ad \dots \longrightarrow  \left(b \dots, w_1\right)_1.\]
This will not
affect $w_1$. We notice that with this we are already using $2$ letters of $w$. However, we are also already gaining at least one
letter in $w_0$. So in total, we need $2 \cdot \left(l\left(w_0\right) + l\left(w_1\right)\right)$ letters to write any words
$w_0, w_1$. 
\end{proof}

We can now provide a better estimate for the bounds of the conjugacy growth function of the Grigorchuk group.

\begin{theorem}\label{thm_grig_conj}
The conjugacy growth function $f(n)$ of the Grigorchuk group $\Gamma$ satisfies
\[e^{n^{0.5}} \precsim f(n) \precsim e^{n^{0.767}}.\]
\end{theorem}

\begin{proof}
The upper bound is given by the word growth as computed by L. Bartholdi in \cite{bartholdi_upper}. This bound is not known to be
sharp. For the lower bound, we look
at the action of an element $g \in \Gamma$ on the two subtrees of level $1$. Let $g$ be of length $n$, and let $h$ be conjugate to
$g$, of length less than or equal to $n$. Let $g$ act as $g_0$ and $g_1$ on the two subtrees $T_0$ respectively $T_1$ of level
$1$ and $h$ as $h_0$ respectively $h_1$:
\[g=\left(g_0,g_1\right)_1, \quad h = \left(h_0, h_1\right)_1.\]
First assume that $g$ fixes the first level, in other words, it contains even many times the generator $a$. Of course $h$ must
then have the same property. In that case, $h$ can only be conjugate to $g$ if we have that either $g_0$ is conjugate to
$h_0$
and $g_1$ to $h_1$, or we have that $g_0$ is conjugate to $h_1$ and $g_1$ to $h_0$. 

\medskip

By Lemma \ref{lemma_subwords} we can, with a
word of length $n$, write at least any combination of words of lengths $\frac n4$ on the two top subtrees. So we can produce at
least as many different classes of conjugates of words of length $n$ as we can have different conjugates of words of length
$\frac n4$ on each subtree, divided by $2$ since the two subtrees can be interchanged by conjugation with $a$. We can express
this recursively, where we have another factor $1/T$ because we only count words which lie in the first level stabilizer:
\[f(4n) \geq \frac1{2T} f(n)^2.\]
We repeat this to get
\[f\left(4^{i}n\right) \geq \frac1{(2T)^{1+2+4+\dots+2^{i-1}}} f(n)^{2^{i}}\] and as in the proof of Theorem
\ref{thm_conj_branch} we obtain \[f\left(4^in\right) \geq \left(\frac{f(n)}{2T}\right)^{2^i}.\] Now choosing $n=N$ such that
$f(N)>2T$, which is possible by Lemma
\ref{lemma_conj_unbounded}, and a variable substitution $k=4^i$ then gives
\[f(kN) \geq C^{k^{0.5}}, \] for a constant $C>1$.
So we get
\[f(k) \succsim e^{k^{0.5}}.\]
\end{proof}

It appears as if it might be possible to further optimize Lemma \ref{lemma_subwords} with a similar, lengthly and
technical approach as in
\cite{brieussel_thesis}. Applied to Theorem \ref{thm_grig_conj} this could result in $\sigma$ slightly greater than $0.5$.

\medskip

An interesting question informally asked by M. Sapir is whether there exist groups which have oscillating word growth, but
non-oscillating conjugacy growth. In particular, one source of examples of groups with oscillating word growth is given by
examples of L. Bartholdi and A. Erschler in \cite{erschler_osc}.

\begin{question}
Do the groups of oscillating intermediate growth as defined in \cite{erschler_osc} also have oscillating conjugacy
growth?
\end{question}

We expect that the quotient $q(n)=\gamma(n)/f(n)$ of the word growth $\gamma(n)$ and the conjugacy growth $f(n)$ grows
very slowly for most branch groups.

\begin{question}
What can be said about the quotient $\gamma_{\Gamma}(n)/f(n)$ for the Grigorchuk group $\Gamma$ or for branch groups in general?
\end{question}

We emphasize, that even though the construction in \cite{erschler_osc} is using the Grigorchuk group, the resulting groups are no
longer branch.

\section{Conjugacy Width}

The aim of this section is to prove that every element $g \in G$, where $G$ is from a certain class of branch groups, can be
written as a product of
uniformly boundedly many
conjugates of the generators. We call this property bounded conjugacy width (BCW). 

\medskip

We start with a general discussion about BCW and will see how it relates to other algebraic properties. This will emphasize, why
some of the conditions we set for branch groups to have BCW are necessary. At the end, we draw a connection to the
palindromic width of a group and deduce that the Grigorchuk group has finite palindromic width for all generating sets.

\subsection{First results about BCW}\label{subsec_bcw}

In this subsection we discuss groups which have, or do not have, bounded conjugacy width. We first show that
bounded conjugacy width implies finite commutator width. The converse however is not true, we will give examples of
groups which have finite commutator
width but unbounded conjugacy width. In fact, we will establish that no infinite group of polynomial growth can have BCW. We then
show that BCW passes on to finite extensions, but we will give an example that it does not pass on to finite index subgroups.
Further, we will prove that any group with only finitely many conjugacy classes has
BCW and that BCW implies that the abelianisation of the group is finite. These are fairly straight-forward observations and we
list them and sketch the proofs for completeness.

\medskip

It is obvious that having BCW is independent of the chosen generating set. The following proposition says that if a group has bounded
conjugacy width then it has finite commutator width.

\begin{prop}\label{prop_conj_comm}
If a group $H$, generated by a minimal set of generators $X=\left\{x_1, \dots, x_k\right\}$,
has bounded
conjugacy width $N$, then it has finite commutator
width at most $3N$.
\end{prop}

\begin{proof}
Assume an element $h \in H'$ is of the form
\begin{equation}\label{eq_comm_H} h = x_{i_1}^{\rho_1} \cdots x_{i_n}^{\rho_n},\end{equation} where the $x_{i_j} \in X$ for the
generating set $X$ and $\rho_i \in H$, $n \leq N$. We
complete the product with
\[h = \left(\prod_{j=1}^{n} x_{i_j}^{-1} x_{i_j}\right) h = \left(\prod_{j=1}^{n} x_{i_j}^{-1} x_{i_j}\right)
\prod_{j=1}^n x_{i_j}^{\rho_j}.\] In order to write the expression as a product of commutators, we shift the factors
$x_{i_j}^{-1}$
from the left side into the product on the right. We demonstrate this for the first factor $x_{i_1}^{-1}$:
\[h = x_{i_1} \cdot \left(\prod_{j=2}^{n} x_{i_j}^{-1} x_{i_j}\right) \left[x_{i_1} \cdot
\left(\prod_{j=2}^{n} x_{i_j}^{-1} x_{i_j}\right), x_{i_1}^{-1} \right] \cdot  x_{i_1}^{-1} x_{i_1}^{\rho_1} \cdot \prod_{j=2}^n
x_{i_j}^{\rho_j}\]
\[=x_{i_1} \cdot \left(\prod_{j=2}^{n} x_{i_j}^{-1} x_{i_j}\right) \left[x_{i_1} \cdot \left(\prod_{j=2}^{n}
x_{i_j}^{-1} x_{i_j}\right), x_{i_1}^{-1} \right] \cdot \left[x_{i_1},\rho_1\right] \cdot \prod_{j=2}^n x_{i_j}^{\rho_j}.\]
We are
now left to move $n-1$ factors $x_{i_j}^{-1}$ into the product on the far right. One can see that repeating this procedure will
result in a term composed of $\prod_{j=1}^n x_{i_j} \cdot r$, where $r$ is a product of $2n$ commutators. We can now express the
first few terms $z=\prod_{j=1}^n x_{i_j}$ as an element of $H/H' \cdot H'$, hence we will get a product $z=x_1^{\zeta_1} \cdots x_k^{\zeta_k} \cdot
f$, where $f$ is a product of at most $n$ commutators and $\zeta_i \in \mathbb{Z}$. By assumption $h$ was in $H'$, hence the first
term $\prod_{i=1}^k
x_i^{\zeta_i}$ is equal to $1$. In total we obtain a commutator width of at most $3N$.
\end{proof}

We will now prove that no infinite nilpotent group can have bounded conjugacy width. By a result of M. Gromov
(\cite{gromov_nil_poly}), this says that no group of non-constant polynomial growth can have BCW. It is known that all nilpotent
groups have finite commutator width from P. Stroud's thesis \cite{stroud}, hence BCW is a stronger property than
finite commutator width. This in particular implies that the converse of Proposition \ref{prop_conj_comm} is not true. 

\medskip

We first need the following observations, which we will then apply to nilpotent groups.

\begin{lemma}\label{lem_col}
\begin{enumerate}
 \item \label{lem_col_1} If a finitely generated group $G$ has BCW, then its abelianisation $G/G'$ is finite.
 \item \label{lem_col_2} If $G$ is a finitely generated nilpotent group with finite abelianisation $G/G'$, then $G$ is finite.
 \item \label{lem_col_3} Let $G$ be a finitely generated infinite nilpotent group. Then $G$ does not have bounded conjugacy width.
\end{enumerate}

\end{lemma}

\begin{proof}
\eqref{lem_col_1}: Obviously BCW passes to quotients and no infinite abelian group can have BCW. This shows that $G/G'$ must be
finite.
\eqref{lem_col_2}: See \cite[p. 13, Corollary 9]{dan_polycyclic}.
\eqref{lem_col_3}: Assume $G$ had BCW. Then by part (\ref{lem_col_1}) it must have finite abelianisation. However,
(\ref{lem_col_2}) implies that $G$ is finite, contradicting the assumption that $G$ is a finitely generated infinite nilpotent
group.
\end{proof}

As an application of this we can show that BCW is a stronger property than finite commutator width.

\begin{theorem}\label{thm_conj_comm}
Any finitely generated infinite nilpotent group has finite commutator width but has unbounded conjugacy width. 
\end{theorem}

\begin{proof}
Work by P. Stroud in his thesis \cite{stroud} shows that every finitely generated nilpotent group has
finite commutator
width. Any infinite nilpotent group is an example of a group which has finite commutator width but does not have BCW by
Lemma \ref{lem_col}(\ref{lem_col_3}).
\end{proof}

The following lemma applies to the groups constructed by V. Ivanov (\cite{olshanskii_defining_rel}) or by D. Osin in
\cite{osin_smallCancellation}. 

\begin{lemma}\label{lemma_conj_classes}
Assume that a finitely generated group $G$ has only $n$ conjugacy classes. Then it has bounded conjugacy width.
\end{lemma}

\begin{proof}
Take for each conjugacy class a representative of shortest length.
Because we only have $n$ conjugacy classes, we can take the maximum over the lengths of these representatives, denoted by $M$.
Then
it follows that each element of $G$ is a product of at most $M$ conjugates of the generators.
\end{proof}

Together with Proposition \ref{prop_conj_comm} and Lemma \ref{lem_col}(\ref{lem_col_1}) this implies that these groups have finite
commutator width and finite abelianisation. 

\begin{theorem}\label{thm_finite_ext}
If a finitely generated group $H$ has BCW and $G$ is a finite extension of $H$, then $G$ has BCW.
\end{theorem}

\begin{proof}
Let $H=\left<h_1, \dots, h_m\right>$ be a finite index subgroup of $G$, such that $H$ has BCW. Let $M$ be the maximal length of
the generators of $H$ with respect to the finite generating set $X=\left\{x_1, \dots, x_n\right\}$ of $G$. Then by assumption,
every
element of $H$ can be written as
\[h = \prod_{i=1}^K h_i^{t_i} = \prod_{i=1}^K \left(\prod_{j=1}^M x_{k_{i,j}}\right)^{t_i} = \prod_{i=1}^{KM} x_{s_i}^{t_i},
\quad t_i \in H.\] The
latter is
a finite product of conjugates of elements from $X$. Every element $g \in G$ can be written as $g = f \cdot h$ for $f \in G/N, h
\in N$,
where $N = \bigcap_{g \in H} H^g$, where again $N$ has finite index in $G$ because $H$ has. If we take for $f$ the coset
representative of shortest length, then we can denote the maximum over all lengths of such minimal coset representatives by
$T$. It is then clear that every element of $H$ is a product of at most $T+KM$ conjugates of the generators $\left\{x_1, \dots,
x_n\right\}$.
\end{proof}

This implies that the group of exponential conjugacy growth constructed by M. Hull and D. Osin in \cite{hull_osin} with the index
$2$
subgroup with $2$ conjugacy classes has finite commutator width and finite abelianisation. On the other hand, the following
example shows that there exist groups with bounded conjugacy width which have a finite index subgroup that has unbounded
conjugacy width.

\begin{example}\label{ex_dihedreal}
Let $G=\left<r,s \h | \h s^2=r^2=1\right>$ be the infinite dihedreal group. Then the infinite abelian subgroup $\left<sr\right>$
has index $2$. On the other hand, every element of $G$ can be written as a product of at most $2$ conjugates of $r$ and $s$.
\end{example}

\subsection{Certain branch groups}

The aim of this subsection is to demonstrate that if a branch group is just infinite, contains a rooted automorphism and  has
finite commutator width, then every element can be written as a product of uniformly boundedly many conjugates of the generators. 

\begin{lemma}\label{lemma_comm_K}
Assume that $G$ is a branch group that contains a rooted element and has finite commutator width. Let $H_i, L_i$ be its
branching structure as defined above. Then every element of the form $\left([\kappa,\sigma],1\right)_1, \kappa, \sigma \in L_1$ is
a
product of
at most $4 \cdot M$ conjugates of the generators of $G$, where $M$ is the length of the shortest rooted element contained in $G$.
\end{lemma}

\begin{proof}
We aim to express a commutator $\left(\left[\kappa_1, \kappa_2\right], 1, \dots, 1
\right)_1 \in \left(L_1',1\right)_1$ as a product of conjugates of the generators. By assumption there exists a rooted element
$x$.
Without
loss of generality we can assume that $x$ acts in such a way that it moves the leftmost top subtree to the second leftmost top
subtree. Choose
$\kappa=\left(\kappa_1^{-1},
1, \dots, 1\right)_1$ with $\kappa_1 \in L_1$. Then
\[t= x^\kappa x^{-1} = \left(\kappa_1, \kappa_1^{-1}, 1, \dots, 1\right)_1.\] We now proceed and conjugate $t$ with $\lambda =
\left(\kappa_2, 1, \dots, 1 \right)_1$ to get
\[t^{-1} t^\lambda = \left(\kappa_1^{-1} \kappa_1^{\kappa_2},1, \dots, 1\right) = \left(\left[\kappa_1, \kappa_2\right], 1,
\dots,
1\right)_1.\] We see that this is a product of $4$ conjugates of the element $x$, hence of $4\cdot l(x)$ conjugates of the
generators of
$G$.
\end{proof}

\begin{lemma}\label{lemma_comm_G}
Let $G$ be a just infinite branch group that contains at least one rooted element. Then there exists a number $T$ such that every
commutator of the form $[\alpha, \beta]$ for $\alpha, \beta \in G$ is a product of at most $T$ conjugates of the generators of
$G$.

The bound $T$ can be explicitly expressed as $4M+2S$, where $M$ is the maximal length of a minimal coset representative of the
branching subgroup $H_1$ of $G$ and $S$ is the constant coming from Lemma \ref{lemma_comm_K}.
\end{lemma}

\begin{proof}
If $[\gamma,\xi]$ is a commutator with $\gamma,\xi \in G$, then we can write $\gamma,\xi$ as $\gamma=\sigma \kappa, \xi=\tau
\lambda$ with $\kappa,\lambda \in H_1$. Because $H_1$ has finite
index in $G$, there are only finitely many minimal choices for $\sigma,\tau$, hence their length is uniformly bounded over
all
elements. Denote the maximal length of a coset representative by $M$ and denote $\sigma=x_1 \cdots x_n$, $\tau=y_1\dots y_n$ for
$n \leq M$, where the $x_i$ and $y_i$ are some generators of $G$. The
commutator $[\gamma,\xi]$ can with the help of basic commutator identities be written as
\[[\gamma,\xi] = [\sigma \kappa, \tau \lambda] = [\sigma,\tau \lambda]^{\kappa} [\kappa,\tau \lambda] = [x_1\dots
x_n,\tau \lambda]^{\kappa}[\kappa,\lambda][\kappa,\tau]^{\lambda}\]\[=[x_1,\tau \lambda]^{\zeta_1}[x_2\dots
x_n,\tau \lambda]^{\kappa}
[\kappa,\lambda][\kappa,y_1\dots y_n]^{\lambda}=\left(\prod_{i=1}^n \left[x_i,\tau \lambda\right]^{\zeta_i}\right)
[\kappa,\lambda] \left(\prod_{i=1}^n \left[\kappa, y_{n-i+1}\right]^{\eta_i}\right),\] where $\zeta_i=\kappa \cdot \prod_{j=i+1}^n
x_i$ and $\eta_i=\lambda \cdot \prod_{j=n-i+2}^{n} y_i$. By assumption we have $\kappa, \lambda \in H_1=L_1\times L_1$, so $[\kappa,
\lambda] \in (L_1',L_1')_1$ and it actually has the form
$\left([\kappa_0,\lambda_0],[\kappa_1, \lambda_1]\right)_1$. By Lemma \ref{lemma_comm_K} there exists a number $t$ such that the
commutator in each component is a product of at most $t$ conjugates of the generators. Each commutator of the form
$\left[x_i,\rho\right]$ for some $\rho \in G$ is a product of $2$ conjugates of $x_i$:
\[\left[x_i, \rho\right] = x_i^{-1} \cdot x_i^{\rho}.\] 

So we get $2\cdot 2n$ conjugates for the commutators and $2 \cdot t$ more for
the commutator $[\kappa, \lambda]$. In total we hence need $4n+2t$ conjugates of the generators to express a commutator of $G$.
\end{proof}

\begin{theorem}\label{thm_conj_width}
Let $G$ be a just infinite branch group that contains a rooted element and that has finite commutator width. Then $G$ has bounded
conjugacy width.
\end{theorem}

\begin{proof}
Because $G$ is just infinite the normal subgroup $G'$ of $G$ has finite index in $G$. Hence every element can be
written in the form $\gamma = \xi \cdot \rho$, for some $\rho \in G'$ and there are only finitely many minimal choices for $\xi$.
Denote by $M$ the length of the longest minimal coset representative. The fact that $G$ has finite commutator width gives us that
$\rho$ is a product of at most $C$ commutators of the form $[x,y], x,y \in G$. By Lemma \ref{lemma_comm_G}, there exists a number
$T$ such that each of them is a product of at most $T$ conjugates of the generators of $G$. Hence every element of $G$ is a
product of at most $M + C \cdot T$ conjugates of the generators of $G$.
\end{proof}

We can see from Proposition \ref{prop_conj_comm} that the condition of having finite commutator width is necessary. However,
proving that a group has BCW would also provide an effective way to prove that it has finite commutator width. To see that the
condition to
be just infinite is necessary, we need another theorem.

\begin{theorem}\cite{grigor_bg}
A branch group $G$ is just infinite if and only if all $H_i'$ have finite index in $H_i$.
\end{theorem}

In particular, this implies that $G$ cannot be just infinite if $H_0=G$ does not have finite abelianisation. Hence we obtain
that the condition to be just infinite cannot
be omitted.

\begin{cor}
There exist branch groups which do not have BCW.
\end{cor}

\begin{proof}
The groups studied by the author in \cite{me_branchgroups} are not just infinite, hence cannot have BCW. 
\end{proof}

\medskip

At this moment, the Grigorchuk group is the only known branch group that satisfies all hypotheses of Theorem
\ref{thm_conj_width}. In particular it has been shown to have finite commutator width by I. Lysenok, A. Miasnikov and A. Ushakov
(\cite{ushakov}).

\begin{cor}
The Grigorchuk group has bounded conjugacy width.
\end{cor}

Computer experiments of L. Bartholdi have suggested that for the Grigorchuk group we might in fact have that every element of
$\Gamma'$ is a product of only four conjugates of the generator $a$. This would in particular imply that its commutator width is
at
most $2$ because \[a^xa^y = x^{-1}xy^{-1} [xy^{-1},a ]aa y = [xy^{-1},a]^y,\] which uses in particular that $a^2=1$.
This leads to the following open question:

\begin{question}
Does the Grigorchuk group $\Gamma$ have commutator width $2$? 
\end{question}

Because of the above computer experiments, if there exists an element $g \in \Gamma'$ which is not a product of $2$ commutators,
then
its length in the standard generators must be at least $17$.

\subsection{Palindromes}

A \emph{palindrome} is a group word which reads the same left-to-right as right-to-left. It has been studied over the last decade
by various authors whether a group has the property that every element is a product of uniformly boundedly many palindromes,
see \cite{bardakov_nilpotent}, \cite{bardakov_soluble}, \cite{bardakov}, \cite{me_palindromic}, \cite{rileySale}. This notion is
not
known to be independent of the generating set and many examples depend on a specific generating set. In some cases,
the question of bounded conjugacy width and finite palindromic width coincide:

\begin{lemma}
If a group $G$ has a generating set $X=\left\{x_1, \dots, x_n\right\}$ where every $x_i^2=1$, then the palindromes in $G$ with
respect to this generating set are exactly the conjugates of the generators $x_i$.
\end{lemma}

\begin{proof}
We note that if every generator has order $2$, then taking inverses amounts to writing a word backwards.
\end{proof}

It has been
shown by A. Thom and the author in \cite{me_andreas}, that if a group is just infinite, then after a possibly slight
modification of
the generating set, it will have finite palindromic width. This modification in particular rules out that every generator has
order $2$. Here we prove that the Grigorchuk
group has
bounded conjugacy width, hence together with the result from
\cite{me_andreas} it follows that

\begin{cor}
The Grigorchuk group has finite palindromic width with respect to all generating sets.
\end{cor}

For a more detailed study of palindromic width we recommend any of the papers mentioned above.

\renewcommand{\bibname}{References}
\bibliography{../Bibliography/bibliography}        

\begin{thebibliography}{10}

\bibitem{babenko}
I.~K. Babenko.
\newblock Closed geodesics, asymptotic volume and the characteristics of growth
  of groups.
\newblock {\em Izv. Akad. Nauk SSSR Ser. Mat.}, 52(4):675--711, 1988.

\bibitem{bardakov_nilpotent}
V.~Bardakov and K.~Gongopadhyay.
\newblock On palindromic width of certain extensions and quotients of free
  nilpotent groups.
\newblock {\em Internat. J. Algebra Comput.}, 24(5):553--567, 2014.

\bibitem{bardakov_soluble}
V.~Bardakov and K.~Gongopadhyay.
\newblock Palindromic width of finitely generated soluble groups.
\newblock {\em http://arxiv.org/abs/1402.6115}, 2014.

\bibitem{bardakov}
V.~Bardakov and K.~Gongopadhyay.
\newblock Palindromic width of nilpotent groups.
\newblock {\em Journal of Algebra}, 402:379--391, 2014.

\bibitem{bartholdi_upper}
L.~Bartholdi.
\newblock The growth of {G}rigorchuk's torsion group.
\newblock {\em Internat. Math. Res. Notices}, 20, 1998.

\bibitem{erschler_osc}
L.~Bartholdi and A.~Erschler.
\newblock Groups of given intermediate word growth.
\newblock {\em arXiv:1110.3650}, 2011.

\bibitem{Bartholdi_bg}
L.~Bartholdi, R.~I. Grigorchuk, and Z.~\v{S}uni\'k.
\newblock Branch groups.
\newblock {\em Handbook of algebra}, 3:989--1112, 2003.

\bibitem{Bartholdi}
L.~Bartholdi and Z.~\v{S}uni\'k.
\newblock On the word and period growth of some groups of tree automorphisms.
\newblock {\em Comm. Algebra}, 29:4923--4964, 2001.

\bibitem{brandenbursky_gal_kedra_marcinkowski}
M.~Brandenbursky, S.~Gal, J.~Kedra, and M.~Marcinkowski.
\newblock Cancelation norm and the geometry of biinvariant word metrics.
\newblock {\em http://arxiv.org/abs/1310.2921}, 2013.

\bibitem{breuillard_corn}
E.~Breuillard and Y.~Cornulier.
\newblock On conjugacy growth of linear groups.
\newblock {\em Math. Proc. Cambridge Philos. Soc.}, 154(2):261--277, 2013.

\bibitem{brieussel_thesis}
J.~Brieussel.
\newblock {\em Croissance et moyennabilit\'e de certains groupes
  d'automorphismes d'un arbre enracin\'e.}
\newblock Phd-thesis, Universit\'e Diderot Paris 7, 2008.

\bibitem{burago_ivanov_polterovich}
D.~Burago, S.~Ivanov, and Polterovich L.
\newblock Conjugation-invariant norms on groups of geometric origin.
\newblock {\em Adv. Stud. Pure Math.}, 52, 2008.

\bibitem{duszenko}
K.~Duszenko.
\newblock Reflection length in non-affine {C}oxeter groups.
\newblock {\em Bull. Lond. Math. Soc.}, 44(3):571--577, 2012.

\bibitem{me_branchgroups}
E.~Fink.
\newblock A finitely generated branch group of exponential growth without free
  subgroups.
\newblock {\em Journal of Algebra}, 397:625--642, 2014.

\bibitem{me_palindromic}
E.~Fink.
\newblock Palindromic width of wreath products.
\newblock {\em http://arxiv.org/abs/1402.4345}, 2014.

\bibitem{me_andreas}
E.~Fink and A.~Thom.
\newblock Palindromic words in simple groups.
\newblock {\em To appear in the Internat. J. Algebra Comput.}, 2014.

\bibitem{gekhtman}
I.~Gekhtman.
\newblock Stable type of the mapping class group.
\newblock {\em http://arxiv.org/abs/1310.5364}, 2013.

\bibitem{grigor_1}
R.~I. Grigorchuk.
\newblock Degrees of growth of finitely generated groups and the theory of
  invariant means.
\newblock {\em Izv. Akad. Nauk SSSR Ser. Mat.}, 48(5):939--985, 1984.

\bibitem{grigor_bg}
R.~I. Grigorchuk.
\newblock Just infinite branch groups.
\newblock {\em New horizons in pro-p groups, Progr. Math.}, 184:121--179, 2000.

\bibitem{gromov_nil_poly}
M.~Gromov.
\newblock Groups of polynomial growth and expanding maps.
\newblock {\em Inst. Hautes \'Etudes Sci. Publ. Math.}, (53):53--73, 1981.

\bibitem{sapir_guba}
V.~Guba and M.~Sapir.
\newblock On the conjugacy growth functions of groups.
\newblock {\em Illinois J. Math.}, 54(1):301--313, 2010.

\bibitem{hull_osin}
M.~Hull and D.~Osin.
\newblock Conjugacy growth of finitely generated groups.
\newblock {\em Adv. Math.}, 235:361--389, 2013.

\bibitem{Segal_subgroupGrowth}
A.~Lubotzky and D.~Segal.
\newblock {\em Subgroup growth.}
\newblock Number 212 in Progress in Mathematics. Birkhauser Verlag, Basel,
  2003.

\bibitem{ushakov}
I.~Lysenok, A.~Miasnikov, and A.~Ushakov.
\newblock Quadratic equations in the {G}rigorchuk group.
\newblock {\em http://arxiv.org/abs/1304.5579}, 2013.

\bibitem{margulis}
G.A. Margulis.
\newblock Certain applications of ergodic theory to the investigation of
  manifolds of negative curvature.
\newblock {\em Funkcional. Anal. i Prilozen.}, 3:89--90, 1969.

\bibitem{mccammond_peterson}
J.~Mccammond and Peterson T.K.
\newblock Bounding reflection length in affine {C}oxeter groups.
\newblock {\em J. Algebraic Combin.}, 34(4):711--719, 2011.

\bibitem{olshanskii_defining_rel}
A.Y. Ol'shanskii.
\newblock {\em Geometry of defining relations in groups.}
\newblock Kluwer Academic Publisher, 1991.

\bibitem{osin_smallCancellation}
D.~Osin.
\newblock Small cancellations over relatively hyperbolic groups and embedding
  theorems.
\newblock {\em Ann. of Math.}, 172(2):1--39, 2010.

\bibitem{rileySale}
T.~Riley and A.~Sale.
\newblock Palindromic width of wreath products, metabelian groups and solvable
  max-n groups.
\newblock {\em To appear in Groups - Complexity - Cryptology}, 6(2), 2014.

\bibitem{roblin}
T.~Roblin.
\newblock Sur la fonction orbitale des groupes discrets en courbure n\'egative.
\newblock {\em Ann. Inst. Fourier (Grenoble)}, 52:145--151, 2002.

\bibitem{dan_polycyclic}
D.~Segal.
\newblock {\em Polycyclic groups.}
\newblock Number~82 in Cambridge Tracts in Mathematics. Cambridge University
  Press, 1983.

\bibitem{Segal_finiteImages}
D.~Segal.
\newblock The finite images of finitely generated groups.
\newblock {\em Proc. London Math. Soc.}, 82(3):597--613, 2001.

\bibitem{stroud}
P.~Stroud.
\newblock {\em Topics in the theory of verbal subgroups.}
\newblock Phd-thesis, University of Cambridge, 1966.

\end{thebibliography}

\medskip

Elisabeth Fink\\
DMA - ENS\\
45 rue d'Ulm\\
75005, PARIS\\
FRANCE\\
elisabeth.fink@oxon.org\\

\end{document}